\documentclass[11pt]{article}

\usepackage{amsmath, amsthm, amssymb}

\usepackage{fullpage}

\usepackage{hyperref}

\newtheorem{theorem}{Theorem}[section]
\newtheorem{lemma}[theorem]{Lemma}
\newtheorem{claim}[theorem]{Claim}

\newtheorem{corollary}[theorem]{Corollary}
\newtheorem{proposition}[theorem]{Proposition}
\theoremstyle{definition}

\newtheorem{definition}[theorem]{Definition}
\theoremstyle{remark}

\newcommand{\cB}{\mathcal{B}}

\newcommand{\cF}{\mathcal{F}}

\title{$L$-systems and the Lov\'asz number}
\author{William Linz\thanks{University of South Carolina, Columbia, SC. ({\tt wlinz@mailbox.sc.edu}). Partially supported by NSF RTG Grant DMS 2038080. Part of this research was completed while the author was at the University of Illinois and partially supported by NSF RTG Grant DMS-1937241}}
\date{\today}

\begin{document}
\maketitle

\begin{abstract}
Given integers $n > k > 0$, and a set of integers $L \subset [0, k-1]$, an \emph{$L$-system} is a family of sets $\cF \subset \binom{[n]}{k}$ such that $|F \cap F'| \in L$ for distinct $F, F'\in \cF$. $L$-systems correspond to independent sets in a certain generalized Johnson graph $G(n, k, L)$, so that the maximum size of an $L$-system is equivalent to finding the independence number of the graph $G(n, k, L)$. The \emph{Lov\'asz number} $\vartheta(G)$ is a semidefinite programming approximation of the independence number $\alpha$ of a graph $G$. 

In this paper, we determine the leading order term of $\vartheta(G(n, k, L))$ of any generalized Johnson graph with $k$ and $L$ fixed and $n\rightarrow \infty$. As an application of this theorem, we give an explicit construction of a graph $G$ on $n$ vertices with a large gap between the Lov\'asz number and the Shannon capacity $c(G)$. Specifically, we prove that for any $\epsilon > 0$, for infinitely many $n$ there is a generalized Johnson graph $G$ on $n$ vertices which has ratio $\vartheta(G)/c(G) = \Omega(n^{1-\epsilon})$, which improves on all known constructions. The graph $G$ \textit{a fortiori} also has ratio $\vartheta(G)/\alpha(G) = \Omega(n^{1-\epsilon})$, which greatly improves on the best known explicit construction.  
\end{abstract}

\section{Introduction}
For integers $a$ and $b$, we use the notation $[a, b]$ to denote the set $\{a, a+1, \ldots, b\}$. If $a=1$, we abbreviate this to $[b]$. Let $\binom{[n]}{k}$ be the family of all $k$-sets on the ground set $[n]$. For a set of integers $L \subset [0, k-1]$, an \emph{$L$-system} is a family $\cF \subset \binom{[n]}{k}$ such that for any distinct sets $A, B\in \cF$, $|A\cap B| \in L$. 
  
$L$-systems correspond to independent sets in \emph{generalized Johnson graphs}.

\begin{definition}\label{defn:johnsongraph}
Let $n$ and $k$ be positive integers with $n > k$ and $L\subset [0, k-1]$. The \emph{generalized Johnson graph} $G = G(n, k, L)$ is the graph with $V(G) = \binom{[n]}{k}$, and $AB \in E(G) \iff |A\cap B| \notin L$.

\end{definition} 

In particular, the Kneser graphs are the graphs $G(n, k, \{1, 2, \ldots, k-1\})$, while the Johnson graphs are the graphs $G(n, k, \{0, 1, \ldots, k-2\})$. 

The correspondence between $L$-systems and independent sets of generalized Johnson graphs $G(n, k, L)$ implies that if $\cF \subset \binom{[n]}{k}$ is an $L$-system, then $|\cF| \le \alpha(G(n, k, L))$. In this paper, we study the Lov\'asz number of generalized Johnson graphs $G(n, k, L)$. The \emph{Lov\'asz number }$\vartheta(G)$ is a semidefinite programming approximation for the independence number $\alpha(G)$ of a graph $G$. The Lov\'asz number is useful to compute the Shannon capacity $c(G)$\footnote{The Shannon capacity is more typically denoted $\Theta(G)$. We use the notation $c(G)$ to avoid clashes with the asymptotic notation $\Theta$.}. The Shannon capacity~\cite{Sh1956} of a graph $G$ is defined as $c(G) = \lim_{n\rightarrow\infty}(\alpha(G_n))^{1/n}$, where $G_n = G^{\boxtimes n}$ (recall the strong product $\boxtimes$ of two graphs $G$ and $H$ is the graph $G\boxtimes H$ with vertex set $V(G) \times V(H)$, and $(u_1, v_1) \sim (u_2, v_2)$ if and only if $u_1 = u_2$ and $v_1 \sim v_2$ or $u_1 \sim u_2$ and $v_1 = v_2$ or $u_1\sim u_2$ and $v_1\sim v_2$). The Shannon capacity is notoriously hard to compute, as there is no known algorithm which computes the Shannon capacity in a finite amount of time; in fact, there isn't even a known algorithm which can approximate the Shannon capacity to within a constant factor. Lov\'asz~\cite{L} proved that $\alpha(G) \le c(G) \le \vartheta(G)$, so if it can be shown that $\vartheta(G) = \alpha(G)$ (as Lov\'asz~\cite[Theorem 13]{L} showed for the Kneser graph), then the Shannon capacity of the graph $G$ is also determined. Indeed, Schrijver~\cite{Sch1978} (for $n\rightarrow \infty$) and Wilson~\cite{Wilson1984} (for the optimal range $n\ge (t+1)(k-t+1)$) proved the $t$-intersecting Erd\H{o}s-Ko-Rado theorem by showing $\vartheta(G(n, k, [t, k-1])) = \binom{n-t}{k-t}$, so it is natural to wonder how good an upper bound for $\alpha(G(n, k, L))$ can be obtained from $\vartheta(G(n, k, L))$. For example, does $\vartheta(G(n, k, L))$ give the correct order of magnitude for $\alpha(G(n, k, L))$? 

The first main result of this paper is the determination of the leading order term of $\vartheta(G(n, k, L))$ for fixed $k$ and $L$ with $n\rightarrow\infty$. Let $L = \{\ell_1, \ell_2, \ldots, \ell_s\}$ with $\ell_1 < \ell_2 < \ldots < \ell_s$. A set of integers $\{\ell_{m}, \ldots, \ell_{m+p}\} \subset L$ is a \textit{run} of consecutive integers in $L$ if $\ell_{m + i} = \ell_m + i$ for every $0\le i\le p$. The set $\{\ell_m, \ldots, \ell_{m+p}\}$ is a \textit{full run} of consecutive integers in $L$ if additionally the following two conditions hold: (1) either $m=1$ or $\ell_{m-1} < \ell_m - 1$; (2) either $m+p = s$ or $\ell_{m+p+1} > \ell_{m+p} + 1$. The set $L$ may be divided uniquely into full runs of consecutive integers. For example, if $L = \{1, 3, 4, 7, 8, 9, 11\}$, then $L = \{1\} \cup \{3, 4\} \cup \{7, 8, 9\} \cup \{11\}$ contains $4$ full runs of consecutive integers. The length of a run is the cardinality of the set associated with the run. 

\begin{theorem}\label{thm:genlovaszbd}
Let $G = G(n, k, L)$ be a generalized Johnson graph for some integers $n > k > 0$ and a set $L = \{\ell_1, \ell_2, \ell_3, \ldots, \ell_s\} \subset [0, k-1]$ with $\ell_1 < \ell_2 < \ldots < \ell_s$. Suppose $L$ contains $b$ full runs of consecutive integers. Then, there is a constant $c$ depending on $k$ and $L$ but not on $n$ such that 
\[\vartheta(G(n, k, L))\ge \prod_{i=1}^{b}m_i!\cdot\frac{\binom{k}{k-\ell_1}\binom{k-\ell_1-1}{k-\ell_2}\cdots \binom{k-\ell_{s-1}-1}{k-\ell_s}}{\prod_{i=1}^s(\ell_i+1)(k-\ell_i)}n^s -cn^{s-1} ,\]
and a constant $C$ depending on $k$ and $L$, but not on  $n$ such that
\[\vartheta(G(n, k, L))\le \prod_{i=1}^{b}m_i!\cdot\frac{\binom{k}{k-\ell_1}\binom{k-\ell_1-1}{k-\ell_2}\cdots \binom{k-\ell_{s-1}-1}{k-\ell_s}}{\prod_{i=1}^s(\ell_i+1)(k-\ell_i)}n^s +Cn^{s-1},\]
where $m_i$ is the length of the $i$th full run of consecutive integers in $L$ for $1\le i\le b$. 

In particular,  
\[\vartheta(G(n, k, L)) = \Theta(n^{|L|}).\]
\end{theorem}

Because generalized Johnson graphs can be expressed as unions of graphs from the Johnson scheme, their Lov\'asz number can be expressed as the solution of a linear program, as proved by Schrijver~\cite{Sch1979}. In Section 2, we introduce and analyze this linear program in order to prove Theorem~\ref{thm:genlovaszbd}. 

Theorem~\ref{thm:genlovaszbd} shows that $\vartheta(G(n, k, L))$  is of the same order of magnitude as the two general bounds for the maximum size of a $L$-system, due to Deza-Erd\H{o}s-Frankl and Ray-Chaudhuri--Wilson. 

\begin{theorem}[Deza--Erd\H{o}s--Frankl~\cite{DEZ}]\label{dezthm}
If $\cF \subset \binom{[n]}{k}$ is an $L$-system, then for $n > 2^kk^3$
\[|\cF| \le \prod_{\ell \in L}\frac{n-\ell}{k-\ell} = \binom{n}{k} \prod_{\ell' \notin L}\frac{k-\ell'}{n-\ell'}.\]
Furthermore, if $|\cF| \ge 2^{s-1}k^2n^{s-1}$, then 
\[|\cap_{F\in \cF}F| \ge \ell_1,\]
and if $s\ge 2$ and $|\cF| \ge 2^kk^2n^{s-1}$, then 
\[(\ell_2 - \ell_1) | (\ell_3 - \ell_2) | \cdots | (\ell_s - \ell_{s-1}) | (k-\ell_s).\]
\end{theorem}

\begin{theorem}[Ray-Chaudhuri--Wilson~\cite{RCW}]\label{rcwthm}
If $\cF \subset \binom{[n]}{k}$ is an $L$-system, 
then $|\cF| \le \binom{n}{|L|}$.
\end{theorem}

The Deza-Erd\H{o}s-Frankl bound in Theorem~\ref{dezthm} is always at least as strong as the Ray-Chaudhuri--Wilson bound in Theorem~\ref{rcwthm}, but Theorem~\ref{rcwthm} is valid for all $n$. Schrijver~\cite[Equation (54)]{Sch1978} asked if the Lov\'asz number bound is at least as good as the Deza-Erd\H{o}s-Frankl bound; that is, is it the case that
\[\vartheta(G(n, k, L)) \le \prod_{\ell \in L}\frac{n-\ell}{k-\ell},\]
as $n\rightarrow \infty$?  For given $n$, $k$, $L$ and $L^C = [0, k-1] \setminus L$, it follows from an old result of Schrijver (Theorem~\ref{thm:schprodthm} in Section 2 of the present paper) that one either has 
\[\vartheta(G(n, k, L)) \le \prod_{\ell \in L}\frac{n-\ell}{k-\ell}\]
or 
\[\vartheta(G(n, k, L^C)) \le \prod_{\ell \notin L} \frac{n-\ell}{k-\ell},\]

so the Lov\'asz number bound does not do worse than Theorem~\ref{dezthm} for at least one of $L$ and $L^C$. Schrijver's question is equivalent to asking if $\vartheta(G(n, k, L)) = \prod_{\ell \in L}\frac{n-\ell}{k-\ell}$ as $n\rightarrow \infty$. Theorem~\ref{thm:genlovaszbd} gives the leading order term of $\vartheta(G(n, k, L))$ for all $k$ and $L$ and allows one to easily construct many examples of $k$ and $L$ for which $\vartheta(G(n, k, L)) \neq  \prod_{\ell \in L}\frac{n-\ell}{k-\ell}$.


For the cases $|L| = 1$ and $|L| = k-1$, we can state the value of $\vartheta(G(n, k, L))$ more precisely. 

\begin{theorem}\label{thm:exactL=1}
Let $n, k, \ell$ be integers with $n > k$ and $0\le \ell < k$, and let $L = \{\ell\}$. Then, as $n\rightarrow \infty$, 
\[\vartheta(G(n, k, L)) = 1 + \frac{\binom{k}{k-\ell}\binom{n-k}{k-\ell}}{\sum_{j=0}^{k-\ell}(-1)^{j+1}\binom{\ell+1}{j}\binom{k-\ell-1}{k-\ell-j}\binom{n-k-\ell-1}{k-\ell-j}} = \frac{\binom{k}{k-\ell}}{(k-\ell)(\ell+1)} n + O(1).\]
Similarly, if $L' = [0, k-1]\setminus \{\ell\}$, then as $n\rightarrow \infty$,
\[\vartheta(G(n, k, L')) = \frac{\binom{n}{k}}{\vartheta(G(n, k, L)) } = \frac{(\ell+1)(k-\ell)}{k!\binom{k}{k-\ell}}n^{k-1} + O(n^{k-2}).\]
\end{theorem}

Note that $\binom{k}{k-\ell} > (k-\ell)(\ell+1) $ when $k\ge \ell+2$ and  $\ell \ge 3$, so the Lov\'asz number is usually worse than both the Ray-Chaudhuri--Wilson bound and the Deza-Erd\H{o}s-Frankl bound when $|L| = 1$.  

Theorem~\ref{dezthm} implies that if $\cF$ is an $(n, k, L)$-system with $L = \{\ell_1, \ell_2, \ldots, \ell_s\}$, then $|\cF| \le C(k)n^{s-1}$ unless $(\ell_2 - \ell_1) | (\ell_3 - \ell_2) | \cdots | (\ell_s - \ell_{s-1}) | (k-\ell_s)$. Furthermore, F\"{u}redi~\cite{Furedi83} (see also \cite{Furedi91}) proved a general upper bound on the order of magnitude of the maximum size of an $L$-system in terms of the \emph{rank} of an $L$-system, which is defined in terms of so-called intersection structures. Theorem~\ref{thm:genlovaszbd} is therefore a somewhat negative result from the viewpoint of the question of determining the maximum size of an $L$-system, as it implies that the Lov\'asz number cannot improve on the known order of magnitude bounds for $L$-systems as $n\rightarrow\infty$ for any $k$ and $L$. 

On the other hand, we can use Theorem~\ref{thm:genlovaszbd} to give explicit examples of graphs with large gaps between the Shannon capacity and the Lov\'asz number (and \textit{a fortiori} large gaps between the independence number and the Lov\'asz number). The question of how large the ratio $\vartheta(G)/\alpha(G)$ can be for a graph on $n$ vertices is related to the question of how good an approximation $\vartheta(G)$ is for the independence number and the Shannon capacity. Lov\'asz noted (see \cite{K}) that for random graphs $\vartheta(G)/\alpha(G) = \Theta(n^{\frac12}/\log n)$, while Feige \cite{F} gave a randomized construction showing that there is an infinite family of graphs with $\alpha(G) = n^{o(1)}$ and $\vartheta(G) = n^{1-o(1)}$. However, there does not seem to be much known about explicit constructions which give fairly large values of $\vartheta(G)/\alpha(G)$ or $\vartheta(G)/c(G)$. Alon and Kahale~\cite{AK} showed that  for any $\epsilon > 0$, there is a  Frankl-R\"{o}dl graph on $n$ vertices with $\vartheta(G) \ge (\frac12 - \epsilon)n$ and $\alpha(G) = O(n^{\delta})$, where $\delta = \delta(\epsilon) < 1$. Peeters~\cite[Remark 2.1]{Peeters} showed that the symplectic graphs $\text{Sp}(2m, 2)$ have $c(\text{Sp}(2m, 2)) = \log_2(n+1) + 1$, while $\vartheta(\text{Sp}(2m, 2)) = \sqrt{n + 1} + 1$, implying a ratio of $\vartheta/c = \Theta(n^{1/2}/\log n)$ (the number of vertices for these graphs is $n=2^{2m} -1)$. 

Our second main result is that for any $\epsilon > 0$, for infinitely many $n$ there is an explicit construction of a graph $G$ on $n$ vertices with ratio $\vartheta(G)/c(G) = \Omega(n^{1-\epsilon})$. This also implies large ratios for $\vartheta(G)/\alpha(G)$ and $\chi(G)/\vartheta(G^c)$. 

\begin{theorem}\label{thm:lovaszthetaeps}
For any $\epsilon > 0$ and infinitely many $n$, there is an explicit construction of a graph $G$ on $n$ vertices with:
\begin{enumerate}
\item \[\vartheta(G)/\alpha(G) = \Omega(n^{1 - \epsilon}).\]
\item \[\vartheta(G)/c(G) = \Omega(n^{1-\epsilon}).\]
\item \[\chi(G)/\vartheta(G^c) = \Omega(n^{1-\epsilon}).\] 
\end{enumerate}
\end{theorem}

Our explicit construction is a particular generalized Johnson graph. Generalized Johnson graphs have been used, for example, to give explicit constructions of Ramsey graphs~\cite{FW1981} and large gaps between the \textit{minrank} parameter of a graph and the Lov\'asz number~\cite{Haviv, LubStav}. The main technical difference in our construction is that in the generalized Johnson graph we fix $k$ and $L$ and let $n\rightarrow \infty$, while in previous constructions $n$ is dependent on $k$. We give our construction in Section 4; in Section 3, we recall some tools that will be of use in proving Theorem~\ref{thm:lovaszthetaeps}, in particular \emph{the Haemers bound}. 

\section{The Lov\'asz number of generalized Johnson graphs}

\subsection{Background on association schemes}

For completeness, we begin with the definition of an association scheme. For further background, we refer to the papers of Schrijver~\cite{Sch1978, Sch1979} and the thesis of Delsarte~\cite{Del1973}.

\begin{definition}[Association Scheme]\label{defn:assoscheme}
Let $X$ be a finite set, and for an integer $n \ge 1$, let $\mathcal{R} = \{R_0, R_1, \ldots, R_n\}$ be a partition of $X\times{X}$. The pair $(X, \mathcal{R})$ is a symmetric \emph{association scheme} with $n$ classes if 
\begin{enumerate}
\item $R_0 = \{(x, x): x\in X\};$
\item For $0\le k\le n$, \[R_k^{-1}:= \{(y, x): (x, y) \in R_k\} = R_k;\]
\item For any triple of integers $i, j, k = 0, 1, \ldots, n$, there are intersection numbers $p_{ij}^k$ such that, for all $(x, y) \in R_k$, \[|\{z\in X: (x, z)\in R_i,\, (z, y)\in R_j\}| = p_{ij}^k.\]
\end{enumerate}
\end{definition}

Let $|X| = m$. Each pair $(X, R_i)$ may be considered as a graph with vertices from $X$ and edges $xy$ if $(x, y) \in R_i$. An alternative statement of Definition~\ref{defn:assoscheme} is that an association scheme is an edge decomposition of the complete graph $K_m$ into graphs $G_1, \ldots G_n$ such that for $0 \le i, j, k\le n$,  $A_iA_j =  \sum_k p_{ij}^kA_k$, where $A_0 = I$ and $A_i$ is the adjacency matrix of the graph $G_i$ for $1\le i\le n$. 

One particular example of a symmetric association scheme is the Johnson scheme. 

\begin{definition}[Johnson scheme]
The Johnson scheme is a symmetric association scheme $(X, \mathcal{R})$ where $X$ is the set $\binom{[n]}{k}$ for positive integers $k$ and $n$ with $n\ge 2k$, and 
\[R_i = \left\{(A, B) \in X\times X: \frac12|A\triangle B| = i\right\}.\] 

\end{definition}

Generalized Johnson graphs correspond to unions of graphs in the Johnson scheme. 

\subsection{Analyzing the linear program for the Lov\'asz number of generalized Johnson graphs}

We now give a formal definition of the Lov\'asz number of a graph. 

\begin{definition}[Lov\'asz number]\label{defn:lovasznumbergen}
Let $G$ be a graph on $n$ vertices. Let $\cB$ be the set of all $n\times n$ matrices $B = b_{ij}$ such that $B$ is positive semidefinite, $\text{trace}(B) = 1$ and $b_{ij} = 0$ whenever $ij \in E$. Then, 
\[\vartheta(G) = \max_{B\in \cB}\sum_{i, j}b_{ij}.\]

\end{definition}

For general graphs, the Lov\'asz number can only be computed by a semidefinite program. However, as Schrijver~\cite[Theorem 4]{Sch1979} proved, if the graph $G$ is a union of graphs in a symmetric association scheme, then the semidefinite program can be reduced to a linear program. In the following theorem, we first give the general form of the linear program, and then specialize to the specific case of the Johnson scheme on $\binom{[n]}{k}$. 

\begin{theorem}[Schrijver~\cite{Sch1979}]\label{thm:schlp}
Let $(X, \mathcal{R} = \{R_0, \ldots, R_n\})$ be a symmetric association scheme and let $0\in M \subset \{0, \ldots, n\}$. Let $G = (X, E)$ be the graph with edge-set $E = \cup_{i\notin M} R_i$. Then, 
\begin{equation}\label{eqn:schlp}
\vartheta(G) = \max\left\{\sum_{i=0}^na_i : a_0 = 1; a_i =0\, \text{ for } i\notin M; \, \sum_{i=0}^na_iQ_i^u\ge 0\, \text{ for } u=0, 1, \ldots n\right\}.
\end{equation}
where $Q_i^u = \frac{\mu_u}{\nu_i}P_i^u$. For the Johnson scheme on $\binom{[n]}{k}$, 
\[\nu_i = \binom{k}{i}\binom{n-k}{i},\]
\[\mu_u = \binom{n}{u} - \binom{n}{u-1},\]
\[P_i^u = \sum_{j=0}^i (-1)^j \binom{u}{j} \binom{k-u}{i-j}\binom{n-k-u}{i-j}.\]
Note that we define $P_i^0 = \nu_i$. 
\end{theorem} 

We refer to the papers of Schrijver~\cite{Sch1979, Sch1978} for a full explanation of the terms $\mu_u, \nu_i$ and $P_i^u$ for general symmetric association schemes.

We need a theorem of Schrijver~\cite{Sch1979} about the Lov\'asz number of symmetric association schemes, which generalizes a result of Delsarte~\cite{Del1973}. 

\begin{theorem}[Schrijver~\cite{Sch1979}]\label{thm:schprodthm}
Let $G = (V, E)$ be a graph whose edge set is the union of graphs in a symmetric association scheme. Then, 
\[\vartheta(G)\vartheta(G^c) = |V|.\]
\end{theorem}

We first prove Theorem~\ref{thm:exactL=1}. 

\begin{proof}[Proof of Theorem~\ref{thm:exactL=1}]
By Theorem~\ref{thm:schlp}, the Lov\'asz number $\vartheta(G(n, k, L))$ is the solution of the following linear program (with $M = \{0, k-\ell\}$):
\begin{equation}\label{eqn:L=1schlp}
\begin{aligned}
\vartheta(G(n, k, L)) &= \max \quad 1 +a_{k-\ell}\\
 & \text{s.t.}\, a_{k-\ell}\frac{1}{\binom{k}{k-\ell}\binom{n-k}{k-\ell}}\sum_{j=0}^{k-\ell}(-1)^j\binom{u}{j}\binom{k-u}{k-\ell-j}\binom{n-k-u}{k-\ell-j} \ge -1, \, u\in \{0, 1, \ldots, k\}.
\end{aligned}
\end{equation}
To satisfy the $u=\ell+1$ constraint when $n$ is sufficiently large, we need 
\[a_{k-\ell} \le \frac{\binom{k}{k-\ell}\binom{n-k}{k-\ell}}{\sum_{j=0}^{k-\ell}(-1)^{j+1}\binom{\ell+1}{j}\binom{k-\ell-1}{k-\ell-j}\binom{n-k-\ell-1}{k-\ell-j}} = \frac{\binom{k}{k-\ell}}{(k-\ell)(\ell+1)} n + O(1).\]
We set $a_{k-\ell}$ to be the right-hand side of the previous inequality. If $u \le \ell$, then the leading order term of $\sum_{j=0}^{k-\ell}(-1)^j\binom{u}{j}\binom{k-u}{k-\ell-j}\binom{n-k-u}{k-\ell-j}$ is the $j = 0$ term, which implies
\[ \frac{1}{\binom{k}{k-\ell}\binom{n-k}{k-\ell}}\sum_{j=0}^{k-\ell}(-1)^j\binom{u}{j}\binom{k-u}{k-\ell-j}\binom{n-k-u}{k-\ell-j} = \frac{\binom{k-u}{k-\ell}}{\binom{k}{k-\ell}} - O\left(\frac{1}{n}\right) =  O(1), \] so the choice of $a_{k-\ell}$ satisfies these inequalities when $n$ is large enough. 

If $u > \ell + 1$, then the leading order term of $\sum_{j=0}^{k-\ell}(-1)^j\binom{u}{j}\binom{k-u}{k-\ell-j}\binom{n-k-u}{k-\ell-j}$ is the $j = u-\ell$ term, which implies
\[ \frac{1}{\binom{k}{k-\ell}\binom{n-k}{k-\ell}}\sum_{j=0}^{k-\ell}(-1)^j\binom{u}{j}\binom{k-u}{k-\ell-j}\binom{n-k-u}{k-\ell-j} = (-1)^{u-\ell_1}O\left(\frac{1}{n^{u-\ell}}\right), \]
so with the choice of $a_{k-\ell}$, we have 
\[a_{k-\ell} \frac{1}{\binom{k}{k-\ell}\binom{n-k}{k-\ell}}\sum_{j=0}^{k-\ell}(-1)^j\binom{u}{j}\binom{k-u}{k-\ell-j}\binom{n-k-u}{k-\ell-j} = o(1),\]
implying that the inequalities are satisfied as $n\rightarrow \infty$. 

The expression for $\vartheta(G(n, k, L'))$ follows from the expression for $\vartheta(G(n, k, L))$ and  Theorem~\ref{thm:schprodthm}. 
\end{proof}

We now prove Theorem~\ref{thm:genlovaszbd} by analyzing the preceding linear program. In the proof of Theorem~\ref{thm:exactL=1}, the important inequality we needed to satisfy was the $u=\ell_1+1$ inequality. For general $L=\{\ell_1, \ell_2, \ldots, \ell_s\}$, the most important inequalities to satisfy will be the inequalities obtained when $u \in \{\ell_1+1, \ell_2+1, \ldots, \ell_s+1\}$. We note that related arguments for specific $L$-systems were given by Bukh and Cox~\cite[Lemma 12 (2)]{BuCo2019} and by Schrijver~\cite{Sch1978}) in the proof for the $t$-intersecting Erd\H{o}s-Ko-Rado theorem for large $n$. 

\begin{proof}[Proof of Theorem~\ref{thm:genlovaszbd}]
Let $k$ and $L = \{\ell_1, \ell_2, \ldots, \ell_s\}\subseteq[0, k-1]$ be given and fixed, with $n\rightarrow \infty$ and $\ell_1 < \ell_2 < \cdots < \ell_s$. We will show that \[\vartheta(G(n, k, L)) \ge \prod_{i=1}^{b}m_i!\cdot\frac{\binom{k}{k-\ell_1}\binom{k-\ell_1-1}{k-\ell_2}\cdots \binom{k-\ell_{s-1}-1}{k-\ell_s}}{\prod_{i=1}^s(\ell_i+1)(k-\ell_i)} n^{s} - cn^{s-1}\] for some constant $c$ depending on $k$ and $L$ but not on $n$. This is sufficient to prove Theorem~\ref{thm:genlovaszbd}, as the same argument with $L^C = [0, k-1]\setminus\{\ell_1, \ell_2, \ldots, \ell_s\} := \{\ell_1', \ell_2', \ldots, \ell_{k-s}'\}$ implies that \[\vartheta(G(n, k, L^C)) \ge \prod_{i=1}^{b'}m_i!\cdot\frac{\binom{k}{k-\ell_1'}\binom{k-\ell_1'-1}{k-\ell_2'}\cdots \binom{k-\ell_{k-s-1}'-1}{k-\ell_{k-s}'}}{\prod_{i=1}^{k-s}(\ell_i'+1)(k-\ell_i')}n^{k-s} - c'n^{k-s-1},\] so Theorem~\ref{thm:schprodthm} and Claim~\ref{clm:multto1/k!} (proved below) imply that \[\vartheta(G(n, k, L)) = \prod_{i=1}^{b}m_i!\cdot\frac{\binom{k}{k-\ell_1}\binom{k-\ell_1-1}{k-\ell_2}\cdots \binom{k-\ell_{s-1}-1}{k-\ell_s}}{\prod_{i=1}^s(\ell_i+1)(k-\ell_i)}n^{s} + O(n^{s-1}).\]

\begin{claim}\label{clm:multto1/k!}
Let $k$ be a positive integer, let $L = \{\ell_1, \ell_2, \ldots, \ell_s\} \subset [0, k-1]$  with  $\ell_1 < \ell_2 < \cdots < \ell_s$ and set $L^C:=[0, k-1]\setminus L = \{\ell_1', \ell_2', \ldots, \ell_{k-s}'\}$, with $\ell_1' < \ell_2' < \cdots < \ell_{k-s}'$. Suppose $L$ contains $b$ full runs of consecutive integers with lengths $m_i$ for $1\le i \le b$ and $L^C$ contains $b'$ full runs of consecutive integers of lengths $m_j'$ for $1\le j\le b'$. Then, 
\[\prod_{i=1}^{b}m_i!\cdot\frac{\binom{k}{k-\ell_1}\binom{k-\ell_1-1}{k-\ell_2}\cdots \binom{k-\ell_{s-1}-1}{k-\ell_s}}{\prod_{i=1}^s(\ell_i+1)(k-\ell_i)} \cdot \prod_{j=1}^{b'}m_j'!\cdot\frac{\binom{k}{k-\ell_1'}\binom{k-\ell_1'-1}{k-\ell_2'}\cdots \binom{k-\ell_{k-s-1}'-1}{k-\ell_{k-s}'}}{\prod_{i=1}^{k-s}(\ell_i'+1)(k-\ell_i')} = \frac{1}{k!}.\]
\end{claim}

\begin{proof}[Proof of Claim~\ref{clm:multto1/k!}]
Note 
\[\prod_{i=1}^s(\ell_i+1)(k-\ell_i)\prod_{i=1}^{k-s}(\ell_i'+1)(k-\ell_i') = (k!)^2,\]
so it suffices to prove 
\[\prod_{i=1}^{b}m_i!\binom{k}{k-\ell_1}\binom{k-\ell_1-1}{k-\ell_2}\cdots \binom{k-\ell_{s-1}-1}{k-\ell_s}\prod_{j=1}^{b'}m_j'!\binom{k}{k-\ell_1'}\binom{k-\ell_1'-1}{k-\ell_2'}\cdots \binom{k-\ell_{k-s-1}'-1}{k-\ell_{k-s}'} = k!.\]
Now, 
\[\binom{k}{k-\ell_1}\binom{k-\ell_1-1}{k-\ell_2}\cdots \binom{k-\ell_{s-1}-1}{k-\ell_s}\]
\[=\frac{k!}{(k-\ell_1)(k-\ell_2)\cdots (k-\ell_s) \cdot (k-\ell_s-1)!(\ell_s - \ell_{s-1}-1)!\cdots \ell_1!}\]
and similarly
\[\binom{k}{k-\ell_1'}\binom{k-\ell_1'-1}{k-\ell_2'}\cdots \binom{k-\ell_{k-s-1}'-1}{k-\ell_{k-s}'} \]
\[= \frac{k!}{(k-\ell_1')(k-\ell_2')\cdots (k-\ell_{k-s}') \cdot (k-\ell_{k-s}'-1)!(\ell_{k-s}' - \ell_{k-s-1}'-1)!\cdots \ell_1'!} \]
Since $(k-\ell_1)\cdots (k-\ell_s)(k-\ell_{k-s}')\cdots (k-\ell_1') = k!$, it further suffices to prove that 
\[\prod_{i=1}^{b}m_i!\prod_{j=1}^{b'}m_j'! = (k-\ell_s-1)!(\ell_s - \ell_{s-1}-1)!\cdots \ell_1! \cdot (k-\ell_{k-s}'-1)!(\ell_{k-s}' - \ell_{k-s-1}'-1)!\cdots \ell_1'!.\]
We show \[\prod_{i=1}^{b}m_i! = (k-\ell_{k-s}'-1)!(\ell_{k-s}' - \ell_{k-s-1}'-1)!\cdots \ell_1'!\] and
 \[\prod_{i=1}^{b'}m_i'! = (k-\ell_{s}-1)!(\ell_{s} - \ell_{s-1}-1)!\cdots \ell_1!,\] which together complete the proof of the claim.  Indeed, consider two consecutive entries $\ell_i'$ and $\ell_{i+1}'$ in $L'$ (and for the purpose of the following argument, set $\ell_0'=-1$ and $\ell_{k-s+1}' = k$). If $\ell_{i+1}' = \ell_{i}'+1$, then $(\ell_{i+1}'-\ell_{i}'-1)!=0!=1$. Otherwise, the integers  $[\ell_{i}'+1, \ell_{i+1}'-1]$ correspond to a full run of consecutive integers in $L$ of length $(\ell_{i+1}'-1)-(\ell_{i}'+1) +1 = \ell_{i+1}'-\ell_{i}' - 1$, and each full run of consecutive integers in $L$ can be obtained in this way. The same argument with full runs in $L'$ and consecutive entries in $L$ gives the other equality. 
\end{proof}

We have $E(G(n, k, L)) = \{\{X, Y\} : X, Y\in \binom{[n]}{k},  |X\cap Y| \notin \{\ell_1, \ell_2, \ldots, \ell_s\}\}$. We have $|X\cap Y| = t$ if and only if $\frac12 |X\Delta Y| = k-t$, so $XY\in E(G(n, k, L))$ if and only if $XY \notin \cup_{i\in M}R_i$, where $M = \{0, k-\ell_1, \ldots, k-\ell_s\}$. Therefore, by Theorem~\ref{thm:schlp}, the Lov\'asz number $\vartheta(G(n, k, L))$ is the solution of the following linear program (taking $M = \{0, k-\ell_1, \ldots, k-\ell_s\}$):
\begin{equation}\label{eqn:Lschlp}
\begin{aligned}
\vartheta(G(n, k, L)) &= \max \quad 1 + \sum_{i=1}^sa_{k-\ell_i}\\
				& \text{s.t.}\, \sum_{i=1}^sa_{k-\ell_i}\frac{1}{\binom{k}{k-\ell_i}\binom{n-k}{k-\ell_i}}\sum_{j=0}^{k-\ell_i}(-1)^j\binom{u}{j}\binom{k-u}{k-\ell_i-j}\binom{n-k-u}{k-\ell_i-j} \ge -1, \, u\in \{0, 1, \ldots, k\}.
\end{aligned}
\end{equation}

Let $P$ be the $s\times{s}$ matrix which has $(i, j)$-entry $P_{(i, j)} = P_{k-\ell_j}^{\ell_{i}+1}$, so that \[P = \begin{pmatrix}P_{k-\ell_1}^{\ell_1+1} & P_{k-\ell_2}^{\ell_1+1}  & \ldots & P_{k-\ell_s}^{\ell_1+1} \\ \vdots & \vdots & &\vdots \\ P_{k-\ell_1}^{\ell_s+1} & P_{k-\ell_2}^{\ell_s+1} & \ldots  & P_{k-\ell_s}^{\ell_s+1}\end{pmatrix}.\]

Let ${\bf a}$ be the $s\times 1$ column vector with ${\bf a}_{i} = \frac{a_{k-\ell_i}}{\binom{k}{k-\ell_i}\binom{n-k}{k-\ell_i}}$.  Then, the inequalities from \eqref{eqn:Lschlp} with $u\in \{\ell_1+1, \ldots, \ell_s+1\}$ can be stated as 
\[P{\bf a} \ge \begin{pmatrix}-1 \\ -1 \\ \vdots \\ -1\end{pmatrix}.\]

Let ${\bf v}$ be the $(s\times 1)$ column vector with 
\[{\bf v} = P^{-1}\begin{pmatrix}0 \\ 0 \\ \vdots \\ 0 \\ -1\end{pmatrix}.\]
It is not \emph{a priori} clear that such a ${\bf v}$ exists, but we will show subsequently in Lemma~\ref{detPs} that $P$ is invertible as $n\rightarrow \infty$. 
We claim that as $n\rightarrow \infty$, the following is a feasible solution to this linear program:
\begin{equation}\label{eqn:solnLp}
a_{k-\ell_i} = \binom{k}{k-\ell_i}\binom{n-k}{k-\ell_i}{\bf v}_i. 
\end{equation}
By construction, any such solution (if the vector ${\bf v}$ exists) will satisfy the inequalities for $u\in \{\ell_1+1, \ell_2+1, \ldots, \ell_s+1\}$ automatically. We will subsequently verify in Lemma~\ref{lem:otherineqs} that the solution in \eqref{eqn:solnLp} also satisfies the other inequalities as $n\rightarrow \infty$. 

We begin with a simple, but useful observation. 
\begin{claim}\label{clm:Pordmag}
Let $k, L, u$ be as before. As $n\rightarrow \infty$, 
\begin{align*}
P_{k-\ell_i}^u &= \sum_{j=0}^{k-\ell_i}(-1)^j\binom{u}{j}\binom{k-u}{k-\ell_i-j}\binom{n-k-u}{k-\ell_i-j}\\
& = \begin{cases}\frac{1}{(k-\ell_i)!}\binom{k-u}{k-\ell_i}n^{k-\ell_i} - O(n^{k-\ell_i-1}),& u < \ell_i \\ (-1)^{u-\ell_i}\binom{u}{\ell_i}\frac{1}{(k-u)!}n^{k-u} + (-1)^{u-\ell_i+1}O(n^{k-u-1}), & u\ge \ell_i.\end{cases}\\
\end{align*}

\end{claim}
\begin{proof}[Proof of Claim~\ref{clm:Pordmag}]
As $n\rightarrow \infty$, $\binom{n-k-u}{k-\ell_i - j} = \frac{1}{(k-\ell_i-j)!}n^{k-\ell_i-j} - O(n^{k-\ell_i-j-1})$.  Hence, the leading order term of $P_{k-\ell_i}^u$ is the $n^{k-\ell_i-j}$ term, where $j \ge 0$ is the smallest nonnegative integer such that $\binom{k-u}{k-\ell_i - j} > 0$, \textit{i.e.} such that $j \ge u-\ell_i$. The claim follows.
 \end{proof}

Claim~\ref{clm:Pordmag} has two important consequences. First, we use Claim~\ref{clm:Pordmag} to show that $P$ is invertible as $n\rightarrow \infty$. In fact, we determine the leading order term of $\text{det}(P)$. 

\begin{lemma}\label{detPs}
Let $k$ and $L$ be  as before. As $n\rightarrow \infty$, 
\[\text{det}(P) = (-1)^{s}Cn^{sk-(\ell_1+\ell_2+\cdots + \ell_s)-s} + O(n^{sk-(\ell_1+\cdots+\ell_s)-s-1}),\]
where
\[C =  \frac{\prod_{i=1}^s(\ell_i+1)}{ \prod_{i=1}^bm_i!  \prod_{i=1}^s(k-\ell_i-1)!}.\]
In particular, $P$ is invertible as $n\rightarrow \infty$ and the solution \eqref{eqn:solnLp} exists.
\end{lemma}

\begin{proof}[Proof of Lemma~\ref{detPs}]
By Claim~\ref{clm:Pordmag}, we have that
\begin{equation}\label{eqn:Psijleadingterms}
P_{(i, j)} = P_{k-\ell_j}^{\ell_i+1} = \begin{cases}\frac{1}{(k-\ell_j)!}\binom{k-\ell_i-1}{k-\ell_j}n^{k-\ell_j} - O(n^{k-\ell_j-1}),& i < j \\ (-1)^{\ell_i+1-\ell_j}\binom{\ell_i+1}{\ell_j}\frac{1}{(k-\ell_i-1)!}n^{k-\ell_i-1} + (-1)^{\ell_i-\ell_j}O(n^{k-\ell_i-2}), & i \ge j.\end{cases}
\end{equation}
Recall the Leibniz formula for the determinant of an $n\times n$ matrix $A=(a_{ij})_{1\le i, j\le n}$:
\[\text{det}(A) = \sum_{\sigma\in S_n}\text{sgn}(\sigma)\prod_{i=1}^na_{i\sigma(i)}.\]
Recall $\text{sgn}$ is the permutation sign function, so that $\text{sgn}(\sigma) = (-1)^{|\text{inv}(\sigma)|}$, where \[\text{inv}(\sigma) = \{(i, j): i < j,\, \sigma(i) > \sigma(j)\}\] is the set of inversions of $\sigma$. By \eqref{eqn:Psijleadingterms}, in row $i$ the highest order of magnitude term is of the order $n^{k-\ell_{i}-1}$, so the highest order terms of $\text{det}(P)$ are of the order $n^{\sum_{i=1}^s(k-\ell_i-1) } = n^{sk-(\ell_1+\cdots+\ell_s)-s}$. Let $\Sigma$ be the set of all permutations $\sigma \in S_s$ which have the property that \[\prod_{i=1}^sP_{(i,\sigma(i))} = \Theta(n^{sk-(\ell_1+\ell_2+\cdots+\ell_s)-s}).\] These are the highest order terms in the Leibniz formula for $\text{det}(P)$. We show the following result, which will prove Lemma~\ref{detPs}.
\begin{equation}\label{detPshighestorder}
\sum_{\sigma\in\Sigma}\text{sgn}(\sigma)\prod_{i=1}^sP_{(i,\sigma(i))}=(-1)^s\frac{\prod_{i=1}^s(\ell_i+1)}{ \prod_{i=1}^bm_i!  \prod_{i=1}^s(k-\ell_i-1)!}n^{sk-(\sum_{i=1}^s\ell_i)-s} + O(n^{sk-(\sum_{i=1}^s\ell_i)-s-1}).
\end{equation}
We prove a claim showing that the permutations in $\Sigma$ are of a rather restricted type. 
\begin{claim}\label{clm:Sigmarestrict}
Let $\sigma\in \Sigma$. Then, for any $1\le i\le s-1$, $\sigma(i) \le i+1$. Furthermore, if $\ell_i$ is the last integer in its full run of consecutive integers in $L$, then $\sigma(i) \le i$.
\end{claim}
\begin{proof}[Proof of Claim~\ref{clm:Sigmarestrict}]
Let $\sigma \in S_s$ be a permutation with $\sigma(i) \ge i+2$ for some $1\le i\le s-2$. Then, $P_{(i,\sigma(i))} = O(n^{k-\ell_{i+2}}) = o(n^{k-\ell_i-1})$, so the term in the Leibniz product formula corresponding to $\sigma$ will be of an order smaller than $n^{sk-(\ell_1+\ell_2+\cdots+\ell_s)-s}$. 

Similarly, if $\sigma$ is a permutation with $\sigma(i) \ge i+1$, where $\ell_i$ is an integer which is last in its full run of consecutive integers, then $P_{(i,\sigma(i))} = O(n^{k-\ell_{i+1}}) = o(n^{k-\ell_i-1})$, since $\ell_{i+1} \ge \ell_i + 2$, so again the term in the Leibniz product formula corresponding to $\sigma$ will be of smaller order. 
\end{proof}
Claim~\ref{clm:Sigmarestrict} implies that every $\sigma \in \Sigma$ can only permute elements within each full run; that is, there are no $\ell_i$ and $\ell_j$ in different full runs of consecutive integers in $L$ with $\sigma(i)=j$. We will therefore consider runs of consecutive integers in $L$. Let $L_m = \{\ell_{m+1}, \ell_{m+2}, \ldots, \ell_{m+p}\}$ be a (not necessarily full) run of $p$ consecutive integers in $L$ with $\ell_{m+j+1} = \ell_{m+j}+1$ for $1\le j\le p-1$. We prove a result on the terms of the matrix $P$ coming from $L_m$ which will ultimately allow us to prove \eqref{detPshighestorder}. 
\begin{lemma}\label{lem:consecrun}
Let $L_m = \{\ell_{m+1}, \ell_{m+2}, \ldots, \ell_{m+p}\}$ be a run of $p$ consecutive integers in $L$ with $\ell_{m+j+1} = \ell_{m+j}+1$ for $1\le j\le p-1$. Let $\Lambda$ be the set of permutations $\lambda$ on the set $\{m+1, \ldots, m+p\}$ with $\lambda(i)\le i+1$ for $m+1\le i\le m+p-1$. Then, 
\[\sum_{\lambda\in \Lambda}\text{sgn}(\lambda)\prod_{i=m+1}^{m+p}P_{(i,\lambda(i))} = (-1)^p\frac{\prod_{i=1}^p(\ell_{m+i}+1)}{p!\prod_{i=1}^p(k-\ell_{m+i}-1)!}n^{pk-(\sum_{i=1}^p\ell_{m+i})-p} +O(n^{pk-(\sum_{i=1}^p\ell_{m+i})-p-1}). \]
\end{lemma}
\begin{proof}[Proof of Lemma~\ref{lem:consecrun}]
We proceed by strong induction on $p$. The case $p=1$ follows from \eqref{eqn:Psijleadingterms}, as $\lambda(m+1) = m+1$, so $\text{sgn}(\lambda) = 1$ and $P_{(m+1, \lambda(m+1))} = (-1)\frac{\ell_{m+1}+1}{(k-\ell_{m+1}-1)!}n^{k-\ell_{m+1}-1} +O(n^{k-\ell_{m+1}-2})$.  
Now suppose the statement of the lemma holds for all $1\le r\le p-1$; that is, assume that the lemma holds for any run of $r$ consecutive integers $\{\ell_{y}, \ell_{y+1}, \ldots, \ell_{y+r-1}\}$ in $L$. We show that the statement holds for $L_m$ under this assumption. Divide the permutations $\lambda$ in $\Lambda$ up according to the value of  $j$ for which $\lambda(m+j) = m+1$. It follows from the conditions on $\lambda$ that $\lambda(m+i) = m+i+1$ for $1\le i < j$, so that for these permutations $\lambda$, 
\[\prod_{i=m+1}^{m+j}P_{(i, \lambda(i))} =  \binom{\ell_{m+1}+j}{\ell_{m+1}}(-1)^{\ell_{m+j}+1-\ell_{m+1}} \prod_{i=1}^{j}\frac{1}{(k-\ell_{m+i}-1)!}n^{jk - \sum_{i=1}^j\ell_{m+i} - j} + O(n^{jk - \sum_{i=1}^j\ell_{m+i} - j})\] \[ = (-1)^{j}\frac{\prod_{i=1}^j(\ell_{m+i}+1)}{j!\prod_{i=1}^j(k-\ell_{m+i}-1)!}n^{jk - \sum_{i=1}^j\ell_{m+i} - j} + O(n^{jk - \sum_{i=1}^j\ell_{m+i} - j}).\]
There are $j-1$ inversions on $\lambda$ restricted to the set $\{m+1, \ldots, m+j\}$, given by the pairs $(m+1, m+j), \ldots (m+j-1, m+j)$. We now note that there are no inversions $(f, g)$ in $\lambda$ with $m+1\le f\le m+j$ and $m+j+1\le g\le m+p$, so that $|\text{inv}(\lambda)| = j-1 + |\text{inv}(\lambda_{[m+j+1, m+p]})|$, where $\lambda_{[m+j+1, m+p]}$ is the restriction of the permutation $\lambda$ to the set $\{m+j+1, \ldots, m+p\}$. Hence, by applying the inductive hypothesis to the set of consecutive integers $\{\ell_{m+j+1}, \ldots, \ell_{m+p}\}$, we obtain that
\[\sum_{\substack{\lambda\in \Lambda \\ \lambda(m+j) = m+1}} \text{sgn}(\lambda_{[m+j+1, m+p]})\prod_{i=m+j+1}^{m+p}P_{(i,\lambda(i))}\]
\[ = (-1)^{p-j}\frac{1}{(p-j)!}\frac{\prod_{i=j+1}^p(\ell_{m+i}+1)}{\prod_{i=j+1}^p(k-\ell_{m+i}-1)!}n^{(p-j)k-(\sum_{i=j+1}^p\ell_{m+i})-(p-j)} +O(n^{(p-j)k-(\sum_{i=j+1}^p\ell_{m+i})-(p-j)-1}),\] so that
\[ \sum_{\substack{\lambda\in \Lambda \\ \lambda(m+j) = m+1}} \text{sgn}(\lambda)\prod_{i=m+1}^{m+p}P_{(i,\lambda(i))} \] is equal to
\[(-1)^{p+j-1}\frac{1}{j!(p-j)!} \frac{\prod_{i=1}^p(\ell_{m+i}+1)}{\prod_{i=1}^p(k-\ell_{m+i}-1)!}n^{pk-(\ell_{m+1} + \cdots + \ell_{m+p})-p} +O(n^{pk-(\ell_{m+1} + \cdots + \ell_{m+p})-p-1}).\]
By summing over $1\le j\le p$, we obtain that 

\[\sum_{\lambda\in \Lambda}\text{sgn}(\lambda)\prod_{i=m+1}^{m+p}P_{(i,\lambda(i))}  = \sum_{j=1}^p\sum_{\substack{\lambda\in \Lambda \\ \lambda(m+j) = m+1}} \text{sgn}(\lambda)\prod_{i=m+1}^{m+p}P_{(i,\lambda(i))}\]
is equal to
\[ (-1)^p\left(\sum_{j=1}^p(-1)^{j-1}\frac{1}{j!(p-j)!}\right)\frac{\prod_{i=1}^p(\ell_{m+i}+1)}{\prod_{i=1}^p(k-\ell_{m+i}-1)!}n^{pk-(\sum_{i=1}^p\ell_{m+i})-p} +O(n^{pk-(\sum_{i=1}^p\ell_{m+i})-p-1})\]
\[= (-1)^p\frac{1}{p!}\frac{\prod_{i=1}^p(\ell_{m+i}+1)}{\prod_{i=1}^p(k-\ell_{m+i}-1)!}n^{pk-(\sum_{i=1}^p\ell_{m+i})-p} +O(n^{pk-(\sum_{i=1}^p\ell_{m+i})-p-1}),\]
completing the proof.
\end{proof}
We now deduce \eqref{detPshighestorder}, and thus Lemma~\ref{detPs}, from Lemma~\ref{lem:consecrun}. Denote the $b$ (full) runs of consecutive integers in $L$ by $r_1, \ldots ,r_b$, and the first integer of the $i$th full run by $\ell_{c_i}$, and recall that the length of the $i$th full run is denoted by $m_i$. By Claim~\ref{clm:Sigmarestrict}, each permutation $\sigma \in \Sigma$ permutes only the elements within each full run $r_j$. Furthermore, $|\text{inv}(\sigma)| = |\text{inv}(\sigma_{r_1})| + \cdots + |\text{inv}(\sigma_{r_b})|$, where $\sigma_{r_i}$ denotes the restriction of $\sigma$ to the $i$th full run. Therefore, by Lemma~\ref{lem:consecrun},
\[\sum_{\sigma\in\Sigma}\text{sgn}(\sigma)\prod_{i=1}^sP_{(i,\sigma(i))} = \prod_{i=1}^b\left(\sum_{\sigma_{r_i}}\text{sgn}(\sigma_{r_i})\prod_{a=c_i}^{c_i+m_i-1}P_{(a,\sigma_{r_i}(a))}\right)\]
\[= \prod_{i=1}^b(-1)^{m_i}\frac{\prod_{j=0}^{m_i-1}(\ell_{c_i+j}+1)}{m_i!\prod_{j=0}^{m_i-1}(k-\ell_{c_i+j}-1)!}n^{m_ik-(\sum_{j=0}^{m_i-1}\ell_{c_i+j})-m_i} +O(n^{m_ik-(\sum_{j=0}^{m_i-1}\ell_{c_i+j})-m_i-1}) \] 
\[ = (-1)^s\frac{\prod_{i=1}^s(\ell_i+1)}{ \prod_{i=1}^bm_i!  \prod_{i=1}^s(k-\ell_i-1)!}n^{sk-(\ell_1+\ell_2+\cdots + \ell_s)-s} + O(n^{sk-(\ell_1+\ell_2+\cdots + \ell_s)-s-1})\]
\end{proof}

The second consequence of Claim~\ref{clm:Pordmag} is that since $\binom{n-k}{k-\ell_i} = \frac{1}{(k-\ell_i)!}n^{k-\ell_i} - O(n^{k-\ell_i-1})$ as $n\rightarrow \infty$, if $a_{k-\ell_1}, a_{k-\ell_2}, \ldots, a_{k-\ell_s}$ satisfy the following inequalities \eqref{ell_1eqn}, \eqref{ell_meqn}, and \eqref{ell_seqn}, then they also satisfy \eqref{eqn:Lschlp} for sufficiently large $n$. 
 
\begin{equation}\label{ell_1eqn}
 \sum_{i=1}^s\frac{a_{k-\ell_i}}{\binom{k}{k-\ell_i}}\left(\binom{k-u}{k-\ell_i}-O\left(\frac{1}{n}\right)\right) \ge -1,\, \text{for } 0\le u\le \ell_1,
 \end{equation}
 
 \begin{equation}\label{ell_meqn}
 \sum_{i = 1}^m\frac{a_{k-\ell_i}}{\binom{k}{k-\ell_i}}\left(\binom{u}{\ell_i}\frac{(k-\ell_i)!}{(k-u)!}(-1)^{u-\ell_i}\frac{1}{n^{u-\ell_i}} - O\left(\frac{1}{n^{u+1-\ell_i}}\right)\right)  + \sum_{i=m+1}^s\frac{a_{k-\ell_i}}{\binom{k}{k-\ell_i}}\left(\binom{k-u}{k-\ell_i} -O\left(\frac{1}{n}\right)\right) \ge -1,
 \end{equation}
 
 \begin{equation}\label{ell_seqn}
 \sum_{i=1}^s\frac{a_{k-\ell_i}}{\binom{k}{k-\ell_i}}\left(\binom{u}{\ell_i}\frac{(k-\ell_i)!}{(k-u)!}(-1)^{u-\ell_i}\frac{1}{n^{u-\ell_i}} - O\left(\frac{1}{n^{u+1-\ell_i}}\right)\right) \ge -1,\, \text{for } \ell_s + 1 < u \le k.
 \end{equation}

Here \eqref{ell_meqn} must be satisfied for $\ell_m + 1 < u \le \ell_{m+1}, 1\le m \le s-1$. We show that the claimed solution for $a_{k-\ell_1}, \ldots, a_{k-\ell_s}$ satisfies the inequalities \eqref{ell_1eqn}, \eqref{ell_meqn}, \eqref{ell_seqn} as $n\rightarrow \infty$ for $u\in [0, k] \setminus \{\ell_1 + 1, \ldots, \ell_s+1\}$. We first describe the solution $a_{k-\ell_1}, a_{k-\ell_2}, \ldots, a_{k-\ell_s}$ more precisely.

\begin{lemma}\label{lem:ak-ellsispoly}
For each $1\le i\le s$, let $a_{k-\ell_i}$ be the expression given by \eqref{eqn:solnLp}. Then, $a_{k-\ell_i}$ is a rational expression in $n$ of degree $s-i+1$ with positive leading order term $C_i$, so that
\[a_{k-\ell_i} = C_in^{s-i+1} - O(n^{s-i}).\]
Furthermore, 
\begin{equation}\label{eqn:C1}
C_1 =  \prod_{i=1}^{b}m_i!\cdot\frac{\binom{k}{k-\ell_1}\binom{k-\ell_1-1}{k-\ell_2}\cdots \binom{k-\ell_{s-1}-1}{k-\ell_s}}{\prod_{i=1}^s(\ell_i+1)(k-\ell_i)}.
\end{equation}
\end{lemma}

\begin{proof}[Proof of Lemma~\ref{lem:ak-ellsispoly}]
We need a basic linear algebraic fact: if $A$ is an $n\times{n}$ invertible matrix, then $A^{-1} = \frac{1}{\text{det}(A)} \text{adj}(A)$, where $\text{adj}(A)$ is the adjugate matrix, which has $(i, j)$ entry given by $\text{adj}(A)_{i, j} = (-1)^{i+j}\text{det}(A^{(j, i)})$, where $A^{(j, i)}$ is the matrix obtained by removing the $j$th row and $i$th column from $A$. 

We have $a_{k-\ell_i} = \binom{k}{k-\ell_i}\binom{n-k}{k-\ell_i}{\bf v}_i$, where ${\bf v}$ is the vector \[{\bf v} = P^{-1}\begin{pmatrix}0 \\ 0 \\ \vdots \\ 0 \\ -1\end{pmatrix}.\]
Therefore, we only need to analyze the terms in column $s$ of $P^{-1}$.  We first determine the leading order term of $P_{(1, s)}^{-1}$. We have $P_{(1, s)}^{-1} = \frac{(-1)^{s+1}}{\text{det}(P)}\text{det}(P^{(s, 1)})$. We find the leading order term of $\text{det}(P^{(s, 1)})$ by the Leibniz formula. Let $\sigma$ be a permutation which contributes to the leading order term of  $\text{det}(P^{(s, 1)})$. We claim that $\sigma$ must be the identity permutation, \textit{i.e.}, $\sigma(m) = m$ for $1\le m\le s-1$. By \eqref{eqn:Psijleadingterms}, the leading order term in the first row is $\frac{1}{(k-\ell_2)!}\binom{k-\ell_1-1}{k-\ell_2}n^{k-\ell_2}$, as $P_{(1, i)}^{(s, 1)} = \Theta(n^{k-\ell_{i+1}})$ for $1\le i\le s-1$, so in particular $\sigma(1) = 1$. To show $\sigma(j) = j$ for $j > 1$, observe that if $\sigma(m) = m$ for $1\le m < j$, then $\sigma(j) \ge j$, and again by \eqref{eqn:Psijleadingterms}  we have $P_{(j, i)}^{(s, 1)} = \Theta(n^{k-\ell_{i+1}})$ for $j\le i\le s-1$. Thus, $\sigma(j) = j$, and by induction we must have $\sigma(m) = m$ for $1\le m\le s-1$. Thus, by the Leibniz formula
\begin{align*}
\text{det}(P^{(s, 1)})&= \prod_{m=1}^{s-1}P_{(m, m)}^{(s, 1)}  + O(n^{(s-1)k-\sum_{i=2}^s\ell_i-1})\\
& = \binom{k-\ell_1-1}{k-\ell_2}\cdots\binom{k-\ell_{s-1}-1}{k-\ell_s}\frac{1}{\prod_{i=2}^s(k-\ell_i)!}n^{(s-1)k-\sum_{i=2}^s\ell_i} + O(n^{(s-1)k-\sum_{i=2}^s\ell_i-1}).
\end{align*}
Therefore, by Lemma~\ref{detPs}, as $n\rightarrow \infty$
\[a_{k-\ell_1} = \frac{(-1)\binom{k}{k-\ell_1}\binom{n-k}{k-\ell_1}(-1)^{s+1}\text{det}(P^{(s, 1)})}{\text{det}(P)} \]
\[ = \prod_{i=1}^{b}m_i!\cdot\frac{\binom{k}{k-\ell_1}\binom{k-\ell_1-1}{k-\ell_2}\cdots \binom{k-\ell_{s-1}-1}{k-\ell_s}}{\prod_{i=1}^s(\ell_i+1)(k-\ell_i)}n^{s} - O(n^{s-1}).\]

Now, let $2\le i\le s$. We have $P_{(i, s)}^{-1} = \frac{(-1)^{s+i}}{\text{det}(P)}\text{det}(P^{(s, i)})$, so we need to find the leading order term of $\text{det}(P^{(s, i)})$ by the Leibniz formula.
Let $\Sigma^*$ denote the set of permutations of $[s-1]$ which contribute to the leading order term of the Leibniz formula for $\text{det}(P^{(s, i)})$. We claim that the permutations in $\Sigma^*$ permute the elements in $[1, i-1]$ and the elements in $[i, s-1]$ separately.  

\begin{claim}\label{Ps(s,i)clm}
Let $\sigma \in \Sigma^*$ be a permutation which contributes to the leading order term of $\text{det}(P^{(s, i)})$ in the Leibniz formula. Then, $\sigma(j) \le i-1$ for $1\le j\le i-1$. 
\end{claim} 

\begin{proof}[Proof of Claim~\ref{Ps(s,i)clm}]
Let $\chi$ be a permutation of $[s-1]$ which has $\chi(j) = z$ for some $1\le j\le i-1$ and some $z\ge i$. Then, there is some $m\ge i$ and some $1\le b\le i-1$ such that $\chi(m) = b$. Let $\tau$ be the permutation obtained from $\chi$ by swapping $b$ and $z$, so that $\tau(j) = b$, $\tau(m) = z$ and $\tau(p) = \chi(p)$ for $p\in [s-1]\setminus\{j, m\}$. We show that 
\[\prod_{w=1}^{s-1}P_{(w, \chi(w))}^{(s, i)} = o\left(\prod_{w=1}^{s-1}P_{(w, \tau(w))}^{(s, i)}\right),\]
which implies that $\chi \notin \Sigma^*$. 

Since $\tau$ and $\chi$ agree except on $j$ and $m$, it is enough to compare $P_{(j, \chi(j))}P_{(m, \chi(m))}$ and $P_{(j, \tau(j))}P_{(m, \tau(m))}$. By \eqref{eqn:Psijleadingterms}, we have that $P_{(j, z)}^{(s, i)} = \Theta(n^{k-\ell_{z+1}})$ and $P_{(m, b)}^{(s, i)} = \Theta(n^{k-\ell_m-1})$, so that
\[P_{(j, \chi(j))}P_{(m, \chi(m))} = \Theta(n^{2k-\ell_{z+1}-\ell_m-1}).\]
For the permutation $\tau$, there are four cases: if $j < b$, then $P_{(j, b)}^{(s, i)} = \Theta(n^{k-\ell_b})$; otherwise, $P_{(j, b)}^{(s, i)} = \Theta(n^{k-\ell_j-1})$. Similarly, if $m < z+1$, then $P_{(m, z)}^{(s, i)} = \Theta(n^{k-\ell_{z+1}})$, and otherwise $P_{(m, z)}^{(s, i)} = \Theta(n^{k-\ell_m-1})$. In each of the four cases, there is one term of order $\Theta(n^{k-\ell_{z+1}})$ or order $\Theta(n^{k-\ell_m-1})$ and one term of order $\Theta(n^{k-\ell_b})$ or $\Theta(n^{k-\ell_j-1})$. Since $\max\{j, b\} \le i-1$ and $\min\{m, z\} \ge i$, it follows that $\min\{\ell_m+1, \ell_{z+1}\} > \max\{\ell_b, \ell_j + 1\}$, implying that 
\[P_{(j, \chi(j))}P_{(m, \chi(m))} = o(P_{(j, \tau(j))}P_{(m, \tau(m))}),\]
completing the proof of the claim. 
\end{proof}

For a permutation $\sigma \in \Sigma^*$, let $\sigma_{[1, i-1]}$ be the restriction of $\sigma$ to the set $[1, i-1]$ and let $\sigma_{[i, s-1]}$ be the restriction of $\sigma$ to the set $[i, s-1]$. Note that $|\text{inv}(\sigma)| = |\text{inv}(\sigma_{[1, i-1]})| + |\text{inv}(\sigma_{[i, s-1]})|$, so that
\[\sum_{\sigma\in \Sigma^*}\text{sgn}(\sigma)\prod_{i=1}^{s-1}P_{(i, \sigma(i))} = \left(\sum_{\sigma_{[1, i-1]}}\text{sgn}(\sigma_{[1, i-1]})\prod_{m=1}^{i-1}P_{(m, \sigma(m))}^{(s, i)}\right)\cdot \left(\sum_{\sigma_{[i, s-1]}}\text{sgn}(\sigma_{[i, s-1]})\prod_{m=i}^{s-1}P_{(m, \sigma(m))}^{(s, i)}\right).\]

We first consider the restriction of the permutation $\sigma$ to the set $[1, i-1]$. The $(i-1)\times{(i-1)}$ submatrix $P_{i-1}$ of $P^{(s, i)}$ given by the first $i-1$ rows and the first $i-1$ columns of $P^{(s, i)}$ is equivalent to the matrix obtained from the $L$-system with parameters $k$ and $L_{i-1}:= \{\ell_1, \ldots, \ell_{i-1}\}$:
\[P_{i-1} = \begin{pmatrix}P_{k-\ell_1}^{\ell_1+1} & P_{k-\ell_2}^{\ell_1+1}  & \ldots & P_{k-\ell_{i-1}}^{\ell_1+1} \\ \vdots & \vdots & &\vdots \\ P_{k-\ell_1}^{\ell_{i-1}+1} & P_{k-\ell_2}^{\ell_{i-1}+1} & \ldots  & P_{k-\ell_{i-1}}^{\ell_{i-1}+1}\end{pmatrix}.\]

Hence, by the proof of Lemma~\ref{detPs}, we have that 
\[\sum_{\sigma_{[1, i-1]}}\text{sgn}(\sigma_{[1, i-1]})\prod_{m=1}^{i-1}P_{(m, \sigma(m))}^{(s, i)} = (-1)^{i-1}D_in^{k(i-1) -\sum_{j=1}^{i-1}\ell_j-(i-1)} + O(n^{k(i-1) -\sum_{j=1}^{i-1}\ell_j-(i-1)-1}),\] where 
\[D_i = \frac{\prod_{j=1}^{i-1}(\ell_j+1)}{ \prod_{j=1}^am_j'!  \prod_{j=1}^{i-1}(k-\ell_j-1)!},\]
where for simplicity we have assumed $L_{i-1}$ has $a$ full runs of consecutive integers with lengths $m_j'$ for $1\le j\le a$. 
 
 For the restriction of the permutation $\sigma$ to $[i, s-1]$, by a similar inductive argument as was used for $\text{det}(P^{(s, 1)})$, the only permutation on $[i, s-1]$ which contributes to the leading order term in the Leibniz formula is the identity permutation on $[i, s-1]$, so
  \[\sum_{\sigma_{[i, s-1]}}\text{sgn}(\sigma_{[i, s-1]})\prod_{m=i}^{s-1}P_{(m, \sigma(m))}^{(s, i)} = \prod_{m = i}^{s-1}P_{(m, m)}^{(s, i)} = E_i n^{(s-i)k - (\ell_{i+1}+\ldots+\ell_s)} + O(n^{(s-i)k - (\ell_{i+1}+\ldots+\ell_s)-1}),\] where 
  \[E_i = \binom{k-\ell_{i}-1}{k-\ell_{i+1}}\cdots\binom{k-\ell_{s-1}-1}{k-\ell_s}\frac{1}{\prod_{j=i+1}^s(k-\ell_j)!} .\]
  These calculations give that
  \begin{align*}
  \text{det}(P^{(s, i)}) & = \sum_{\sigma \in \Sigma^*}\text{sgn}(\sigma)\prod_{m=1}^{s-1}P_{(m, \sigma(m))}^{(s, i)} + O(n^{(s-1)k - (\sum_{j=1, j\neq i}^s\ell_j) - (i-1)-1})\\
  &=(-1)^{i-1}D_iE_in^{(s-1)k - (\sum_{j=1, j\neq i}^s\ell_j) - (i-1)} + O(n^{(s-1)k - (\sum_{j=1, j\neq i}^s\ell_j) - (i-1)-1}).
  \end{align*}
  Hence, by Lemma~\ref{detPs},
\[a_{k-\ell_i} = \frac{(-1)\binom{k}{k-\ell_i}\binom{n-k}{k-\ell_i}}{\text{det}(P)}(-1)^{s+i}\text{det}(P^{(s, i)}) = C_i n^{s-i+1} - O(n^{s-i}),\]
where \[C_i = \binom{k}{k-\ell_i}\frac{D_iE_i}{(k-\ell_i)!C}\] is a positive constant (here $C$ is the positive constant in the statement of Lemma~\ref{detPs}).   
\end{proof}

We now show that the $a_{k-\ell_i}$ satisfy the inequalities in \eqref{ell_1eqn}, \eqref{ell_meqn}, \eqref{ell_seqn} as $n\rightarrow \infty$ for $u\in [0, k] \setminus \{\ell_1 + 1, \ldots, \ell_s+1\}$.

\begin{lemma}\label{lem:otherineqs}
The solution \eqref{eqn:solnLp} satisfies the inequalities \eqref{ell_1eqn}, \eqref{ell_meqn}, \eqref{ell_seqn} as $n\rightarrow \infty$ for $u\in [0, k] \setminus \{\ell_1 + 1, \ldots, \ell_s+1\}$. 
\end{lemma}

\begin{proof}[Proof of Lemma~\ref{lem:otherineqs}]
By Lemma~\ref{lem:ak-ellsispoly}, since the leading order term $C_1$ of $a_{k-\ell_1}$ is positive, the inequalities in \eqref{ell_1eqn} are satisfied by $a_{k-\ell_1}, \ldots, a_{k-\ell_s}$ for $n$ sufficiently large. 

The inequalities \eqref{ell_meqn} with $\ell_m+2 \le u\le \ell_{m+1}$ for $1\le m \le s-1$ will be satisfied for large $n$ as the leading order term on the left-hand side is $a_{k-\ell_{m+1}}\frac{\binom{k-u}{k-\ell_m}}{\binom{k}{k-\ell_m}} = C_{m+1}\frac{\binom{k-u}{k-\ell_{m+1}}}{\binom{k}{k-\ell_{m+1}}}n^{s-m} - O(n^{s-m-1})$. 

Finally, for $u\ge \ell_s + 2$, each term in the sum on the left-hand side of \eqref{ell_seqn} is $o(1)$, so these inequalities are satisfied as $n\rightarrow \infty$.
\end{proof}

We now complete the proof of Theorem~\ref{thm:genlovaszbd}. By Lemma~\ref{detPs} and Lemma~\ref{lem:otherineqs}, the solution in \eqref{eqn:solnLp} is a feasible solution to \eqref{eqn:Lschlp}. Hence, by Lemma~\ref{lem:ak-ellsispoly},
\[\vartheta(G(n, k, L)) \ge 1 + \sum_{i=1}^sa_{k-\ell_i} = \prod_{i=1}^{b}m_i!\cdot\frac{\binom{k}{k-\ell_1}\binom{k-\ell_1-1}{k-\ell_2}\cdots \binom{k-\ell_{s-1}-1}{k-\ell_s}}{\prod_{i=1}^s(\ell_i+1)(k-\ell_i)}n^s - O(n^{s-1}).\]
\end{proof}

It would be interesting to strengthen Theorem~\ref{thm:genlovaszbd} by finding the exact value of $\vartheta(G(n, k, L))$ (as Theorem~\ref{thm:exactL=1} does if $|L|=1$ or $|L| = k-1$) as $n\rightarrow\infty$. We have opted for the  computations presented here, as they are sufficient to obtain the order of magnitude of $\vartheta(G(n, k, L))$ and the correct constant in front of the $n^{s}$ term. 

We note that the argument used to prove Theorem~\ref{thm:genlovaszbd} also asymptotically determines the best possible result the Delsarte linear programming bound~\cite{Del1973, Sch1979} can give for the graphs $G(n, k, L)$. Since the graphs $G(n, k, L)$ are unions of graphs in the Johnson scheme, Delsarte's linear programming bound $\sigma(G(n, k, L))$ is equal to Schrijver's variant of the Lov\'asz number~\cite{Sch1979}, which is the same linear program as the linear program for the Lov\'asz number in \eqref{eqn:schlp}, with the additional restriction that the variables $a_{0}, a_{1}, \ldots, a_{n}$ must all be nonnegative. The solution in the proof of Theorem~\ref{thm:genlovaszbd} has all $a_{k-\ell_i}$ nonnegative, so \eqref{eqn:solnLp} is also a feasible solution for Delsarte's linear programming bound, so as $n\rightarrow\infty$,
\[\sigma(G(n, k, L)) \ge \prod_{i=1}^{b}m_i!\cdot\frac{\binom{k}{k-\ell_1}\binom{k-\ell_1-1}{k-\ell_2}\cdots \binom{k-\ell_{s-1}-1}{k-\ell_s}}{\prod_{i=1}^s(\ell_i+1)(k-\ell_i)}n^s - cn^{s-1}\]
On the other hand, $\sigma(G(n, k, L)) \le \vartheta(G(n, k, L))$ by definition. We summarize this discussion in the following corollary. 

\begin{corollary}\label{cor:delsartelp}
Let $G = G(n, k, L)$ be a generalized Johnson graph for some integers $n > k > 0$ and a set $L = \{\ell_1, \ell_2, \ell_3, \ldots, \ell_s\} \subset [0, k-1]$ with $\ell_1 < \ell_2 < \ldots < \ell_s$. Suppose $L$ contains $b$ full runs of consecutive integers. Then, there is a constant $c$ depending on $k$ and $L$, but not on $n$ such that 
\[\sigma(G(n, k, L))\ge \prod_{i=1}^{b}m_i!\cdot\frac{\binom{k}{k-\ell_1}\binom{k-\ell_1-1}{k-\ell_2}\cdots \binom{k-\ell_{s-1}-1}{k-\ell_s}}{\prod_{i=1}^s(\ell_i+1)(k-\ell_i)}n^s -cn^{s-1} ,\]
and a constant $C$ depending on $k$ and $L$ but not on $n$ such that
\[\sigma(G(n, k, L))\le \prod_{i=1}^{b}m_i!\cdot\frac{\binom{k}{k-\ell_1}\binom{k-\ell_1-1}{k-\ell_2}\cdots \binom{k-\ell_{s-1}-1}{k-\ell_s}}{\prod_{i=1}^s(\ell_i+1)(k-\ell_i)}n^s +Cn^{s-1},\]
where $m_i$ is the length of the $i$th full run of consecutive integers in $L$ for $1\le i\le b$. 
\end{corollary}

\section{The Haemers bound and the minrank parameter}

We need a different linear algebraic bound for the Shannon capacity of a graph proven by Haemers~\cite{Haem1978, Haem1979}. Haemers' bound is in terms of the \textit{minrank} parameter of a graph. 

\begin{definition}[Minrank]\label{defn:minrank}
Let $G = (V, E)$ be a graph on $n$ vertices and let $\mathbb{F}$ be a field. The $n\times{n}$ matrix $B=(b_{ij})$ is said to \textit{represent} the graph $G$ over $\mathbb{F}$ if $b_{ii} \neq 0$ for $1\le i \le n$ and $b_{ij} = 0$ whenever $ij \notin E(G)$. The \textit{minrank} of $G$ over $\mathbb{F}$ is defined to be 
\[\text{minrank}_{\mathbb{F}}(G) := \min\{\text{rank}_{\mathbb{F}}(M): M \text{ represents } G \text{ over } \mathbb{F}\}.\]
\end{definition}

Haemers proved that the minrank of $G$ over any field $\mathbb{F}$ is an upper bound on the Shannon capacity of $G$. 

\begin{theorem}[Haemers bound]\label{thm:haembd}
Let $G = (V, E)$ be a graph on $n$ vertices and let $\mathbb{F}$ be a field. Then, 
\[c(G) \le \text{minrank}_{\mathbb{F}}(G).\]
\end{theorem}

%

Haviv~\cite{Haviv} has previously noted that the Haemers bound implies bounds on the \text{minrank} of certain generalized Johnson graphs. Indeed, many of the linear algebraic arguments of Frankl and Wilson~\cite{FW1981} on the maximum size of an $L$-system can in fact more strongly be represented as upper bounds on the Shannon capacity of the associated generalized Johnson graph. We give one example of this strengthening. 

Frankl and Wilson~\cite{FW1981} proved the following theorem using inclusion matrices. 

\begin{theorem}[Frankl-Wilson]\label{FWthm}
Let $\cF \subset \binom{[n]}{k}$ and let $q$ be a prime power. Suppose that for all distinct $F, F'\in \cF$,  $|F \cap F'| \not\equiv k\pmod{q}$. Then, 
\[|\cF| \le \binom{n}{q-1}.\]
\end{theorem}

Let $G_q(n, k)$ be the generalized Johnson graph corresponding to the $L$-system given in Theorem~\ref{FWthm}. Lubetzky and Stav~\cite{LubStav} and Haviv~\cite[Proposition 4.5]{Haviv} have noted that essentially the same proof as that of Frankl and Wilson using inclusion matrices gives the following stronger result. 

\begin{proposition}\label{lshprop}
For the generalized Johnson graph $G_q(n, k)$, if $q$ is a prime power and $q\le k+1$, 
\[c(G_q(n, k)) \le \text{minrank}_{\mathbb{F}_p}(G_q(n, k)) \le \binom{n}{q-1}.\]
\end{proposition}

It is this connection that we will use in giving an explicit construction of a graph with large gap between the Shannon capacity and the Lov\'asz number. Note that Theorem~\ref{thm:exactL=1} is already sufficient to obtain explicit generalized Johnson graphs with ratio $\vartheta/\alpha = \Omega(n^{\frac12 - \epsilon})$, as Frankl and F\"{u}redi~\cite{FF} proved that $\alpha(G(n, k, L')) = \Theta(n^{\max\{k-\ell-1, \ell\}})$, where, as in the statement of Theorem~\ref{thm:exactL=1}, $L' = [0, k-1]\setminus\{\ell\}$.  Furthermore, if $\ell+1$ is a prime power, then by setting $k=2\ell+1$, Proposition~\ref{lshprop} implies that the generalized Johnson graph $G_{q}(n, k)$ has ratio $\vartheta/c = \Omega(n^{\frac12 - \epsilon})$. We next show how to obtain constructions with larger gap.

\section{Explicit constructions of large gaps between the Shannon capacity and the Lov\'asz number}

We now give our explicit construction. 

\begin{proof}[Proof of Theorem~\ref{thm:lovaszthetaeps}] 
We use a slight variation of the Ramsey construction of Frankl and Wilson~\cite{FW1981}. Let $q = p^m$ be a prime power and define $k = q^2 - 1$. Then, consider the generalized Johnson graph $G_q(n, k)$ on $\binom{[n]}{k}$ vertices and with $A\sim B \iff |A\cap B|\equiv -1\pmod{q}$. Proposition~\ref{lshprop} implies
\[\text{minrank}_{\mathbb{F}_p}(G_q(n, k)) \le \binom{n}{q-1},\]
while Theorem~\ref{thm:genlovaszbd} implies 
\[\vartheta(G_q(n, k)) = \Theta(n^{q^2-q}).\]
Therefore, with $N = \binom{n}{q^2-1} = \Theta(n^{q^2-1})$, for $n \rightarrow \infty$
\[\vartheta(G_q(n, k))/\text{minrank}_{\mathbb{F}_p}(G_q(n, k)) = \Omega(N^{\frac{q^2-2q+1}{q^2 - 1}}) = \Omega(N^{1-\frac{2}{q+1}}).\]

Given $\epsilon > 0$, choosing $q$ sufficiently large implies the second part of Theorem~\ref{thm:lovaszthetaeps}. The first part of Theorem~\ref{thm:lovaszthetaeps} is a consequence of the second part of the theorem and the inequality $\alpha(G_q(n, k)) \le c(G_q(n, k))$. The third part follows from the inequality $\chi(G_q(n, k)) \ge \frac{N}{\alpha(G_q(n, k))}= \Omega(n^{q^2-q})$ and the fact that $\vartheta(G_q(n, k)^C) = \Theta(n^{q-1})$, which follows from Theorem~\ref{thm:schprodthm}.  
\end{proof}

Lubetzky and Stav~\cite{LubStav} provided a construction of a generalized Johnson graph $G$ with ratio $\vartheta(G)/\text{minrank}_{\mathbb{F}_p}(G) = \Omega(n^{\frac12 - o(1)})$ for any prime $p$. 
The construction used in Theorem~\ref{thm:lovaszthetaeps} gives an explicit construction which greatly improves on this result.

\begin{theorem}\label{thm:fixpcons}
For any $\epsilon > 0$ and any prime $p$, for infinitely many $n$ there is an explicit construction of a graph $G$ on $n$ vertices such that \[\vartheta(G)/\text{minrank}_{\mathbb{F}_p}(G) = \Omega(n^{1-\epsilon}).\]
\end{theorem}


Taking complements of the graph in Theorem~\ref{thm:fixpcons} yields an explicit construction of a graph $G$ with \[\text{minrank}_{\mathbb{F}_{p}}(G)/\vartheta(G) = \Omega(n^{1-\epsilon})\] for any $\epsilon > 0$ and prime $p$, using Theorem~\ref{thm:schprodthm} and the fact~\cite{Peeters} that for any field $\mathbb{F}$ and $n$-vertex graph $G$, $\text{minrank}_{\mathbb{F}}(G) \cdot \text{minrank}_{\mathbb{F}}(G^c)\ge n$.

We finally note that Theorem~\ref{thm:lovaszthetaeps} also holds if the Lov\'asz number is replaced by the Schrijver variant $\sigma$ (or the Delsarte linear programming bound). In particular, the graphs $G_q(n, k)$ again give an explicit construction of a graph $G$ on $N$ vertices with $\sigma(G)/c(G) = \Omega(N^{1-\epsilon})$.

\bibliographystyle{alpha}
\bibliography{lovasztheta}

\begin{thebibliography}{RCW75}

\bibitem[AK98]{AK}
N.~Alon and A.~Kahale.
\newblock Approximating the independence number via the $\vartheta$-function.
\newblock {\em Math. Programming}, 80:253--264, 1998.

\bibitem[BC19]{BuCo2019}
Boris Bukh and Christopher Cox.
\newblock On a fractional version of {H}aemers' bound.
\newblock {\em IEEE Trans. Inform. Theory}, 65(6):3340--3348, 2019.

\bibitem[DEF78]{DEZ}
M.~Deza, P.~Erd\H{o}s, and P.~Frankl.
\newblock Intersection properties of systems of finite sets.
\newblock {\em Proc. London Math. Soc.}, 36:369--384, 1978.

\bibitem[Del73]{Del1973}
P.~Delsarte.
\newblock An algebraic approach to the association schemes of coding theory.
\newblock {\em Philips Res. Rep. Suppl.}, (10):vi+97, 1973.
\newblock Available from
  \url{https://users.wpi.edu/~martin/RESEARCH/philips.pdf}.

\bibitem[F\"83]{Furedi83}
Z.~F\"{u}redi.
\newblock On finite set-systems whose every intersection is a kernel of a star.
\newblock {\em Discrete Math.}, 47(1):129--132, 1983.

\bibitem[F\"91]{Furedi91}
Zolt\'{a}n F\"{u}redi.
\newblock Tur\'{a}n type problems.
\newblock In {\em Surveys in combinatorics, 1991 ({G}uildford, 1991)}, volume
  166 of {\em London Math. Soc. Lecture Note Ser.}, pages 253--300. Cambridge
  Univ. Press, Cambridge, 1991.

\bibitem[Fei97]{F}
U.~Feige.
\newblock Randomized graph products, chromatic numbers, and the {L}ov\'asz
  $\vartheta$-function.
\newblock {\em Combinatorica}, 17:79 -- 90, 1997.

\bibitem[FF85]{FF}
P.~Frankl and Z.~F\"uredi.
\newblock Forbidding just one intersection.
\newblock {\em J. Combin Theory Ser. A}, 39:160 -- 176, 1985.

\bibitem[FW81]{FW1981}
P.~Frankl and R.~M. Wilson.
\newblock Intersection theorems with geometric consequences.
\newblock {\em Combinatorica}, 1(4):357--368, 1981.

\bibitem[Hae79]{Haem1979}
Willem Haemers.
\newblock On some problems of {L}ov\'{a}sz concerning the {S}hannon capacity of
  a graph.
\newblock {\em IEEE Trans. Inform. Theory}, 25(2):231--232, 1979.

\bibitem[Hae81]{Haem1978}
W.~Haemers.
\newblock An upper bound for the {S}hannon capacity of a graph.
\newblock In {\em Algebraic methods in graph theory, {V}ol. {I}, {II}
  ({S}zeged, 1978)}, volume~25 of {\em Colloq. Math. Soc. J\'{a}nos Bolyai},
  pages 267--272. North-Holland, Amsterdam-New York, 1981.

\bibitem[Hav18]{Haviv}
Ishay Haviv.
\newblock On minrank and the {L}ov\'{a}sz theta function.
\newblock In {\em Approximation, randomization, and combinatorial optimization.
  {A}lgorithms and techniques}, volume 116 of {\em LIPIcs. Leibniz Int. Proc.
  Inform.}, pages Art. No. 13, 15. Schloss Dagstuhl. Leibniz-Zent. Inform.,
  Wadern, 2018.

\bibitem[Knu94]{K}
D.~Knuth.
\newblock The {S}andwich {T}heorem.
\newblock {\em Electronic J. Combinatorics}, 1:1--48, 1994.

\bibitem[Lov79]{L}
L.~Lov\'asz.
\newblock On the {S}hannon {C}apacity of a {G}raph.
\newblock {\em IEEE. Trans. Inform. Th.}, pages 1--7, 1979.

\bibitem[LS09]{LubStav}
Eyal Lubetzky and Uri Stav.
\newblock Nonlinear index coding outperforming the linear optimum.
\newblock {\em IEEE Trans. Inform. Theory}, 55(8):3544--3551, 2009.

\bibitem[Pee96]{Peeters}
Ren\'{e} Peeters.
\newblock Orthogonal representations over finite fields and the chromatic
  number of graphs.
\newblock {\em Combinatorica}, 16(3):417--431, 1996.

\bibitem[RCW75]{RCW}
D.~Ray-Chaudhuri and R.M. Wilson.
\newblock On $t$-designs.
\newblock {\em Osaka J. Math.}, 12:737--744, 1975.

\bibitem[Sch79]{Sch1979}
A.~Schrijver.
\newblock A {C}omparison of the {D}elsarte and {L}ov\'asz bounds.
\newblock {\em IEEE Trans. Inform. Theory}, 25(4):425--429, 1979.

\bibitem[Sch81]{Sch1978}
A.~Schrijver.
\newblock Association schemes and the {S}hannon {C}apacity: {E}berlein
  {P}olynomials and the {E}rdos-{K}o-{R}ado theorem.
\newblock In {\em Algebraic methods in graph theory, {V}ol. {I}, {II}
  ({S}zeged, 1978)}, volume~25 of {\em Colloq. Math. Soc. J\'{a}nos Bolyai},
  pages 671--688. North-Holland, Amsterdam-New York, 1981.

\bibitem[Sha56]{Sh1956}
Claude~E. Shannon.
\newblock The zero error capacity of a noisy channel.
\newblock {\em Institute of Radio Engineers Transactions on Information
  Theory}, IT-2(September):8--19, 1956.

\bibitem[Wil84]{Wilson1984}
R.M. Wilson.
\newblock The exact bound in the {E}rdos-{K}o-{R}ado theorem.
\newblock {\em Combinatorica}, 4:247--257, 1984.

\end{thebibliography}

\end{document}